\documentclass[ECP]{ejpecp} % replace ECP by EJP if needed.  
\usepackage{pgf,tikz}
\usetikzlibrary{arrows}

\SHORTTITLE{Optimal Exponent for Coalescence of Geodesics in LPP} 

\TITLE{Optimal Exponent for Coalescence of Finite Geodesics in Exponential Last Passage Percolation}
\AUTHORS{%
  Lingfu~Zhang\footnote{Department of Mathematics, Princeton University, United States of America.
    \EMAIL{lingfuz@math.princeton.edu}}
 }%AUTHORS
 % Separate items with ;

\AMSSUBJ{60K35; 82C23} % Edit. Separate items with ;
\SUBMITTED{March 18, 2020} % Edit.
\ACCEPTED{October 5, 2020} % Edit.

\VOLUME{0}
\YEAR{2012}
\PAPERNUM{0}
\DOI{vVOL-PID}

\ABSTRACT{
In this note, 
we study the model of directed last passage percolation on $\Z^2$, with i.i.d. exponential weight.
We consider the maximum directed paths from vertices $(0,\lfloor k^{2/3}\rfloor)$ and $(\lfloor k^{2/3} \rfloor,0)$ to $(n,n)$, respectively.
For the coalescence point of these paths, 
we show that the probability for it being $>Rk$ far away from the origin is in the order of $R^{-2/3}$.
This is motivated by a recent work of Basu, Sarkar, and Sly \cite{basu2017coalescence}, where the same estimate was obtained for semi-infinite geodesics, and the optimal exponent for the finite case was left open.
Our arguments also apply to other exactly solvable models of last passage percolation.}
\newcommand{\E}{\mathbb E}
\newcommand{\PP}{\mathbb P}
\newcommand{\Z}{\mathbb Z}
\newcommand{\R}{\mathbb R}

\newcommand{\LL}{\mathbb L}

\newcommand{\cA}{\mathcal A}
\newcommand{\cB}{\mathcal B}
\newcommand{\cE}{\mathcal E}
\newcommand{\cC}{\mathcal E}

\newcommand{\cg}{\mathcal G}

\newcommand{\tu}{\tilde u}
\newcommand{\tv}{\tilde v}

\newcommand{\fC}{\mathcal C}

\newcommand{\bn}{\mathbf n}
\newcommand{\boo}{\mathbf 0}
\newcommand{\bk}{\mathbf k}
\newcommand{\bbk}{\overline{\mathbf k}}
\newcommand{\hbk}{\tilde{\mathbf k}}
\DeclareMathOperator{\Exp}{Exp}

\begin{document}
\section{Introduction}
As a model of fluid flow in a random medium, first passage percolation (FPP) has been studied by probabilists for more than fifty years, while many predictions about its geometric structure remain unsettled.
One major prediction is that the planar FPP belongs to the so-called KPZ universality class, proposed by Kardar, Parisi, and Zhang in their seminal work \cite{kardar1986dynamic}. 
While little progress has been made to rigorously establish this prediction for the planar FPP, similar results are known for some exactly solvable directed last passage percolation (DLPP) models, where exact distributional formulas exist and are obtained from combinatorics, representation theory, or random matrix theory.
See e.g. \cite{corwin2012kardar, quastel2014airy} for surveys in this direction.

For general planar FPP, due to the absence of such formulas, its study relies more on understanding of the geodesics (i.e. minimal weight paths between points). In particular, coalescence of geodesics has been wildly used in obtaining geometric information of the model.
The study in this direction was initiated by Newman and co-authors, see e.g. \cite{newman1995surface}.
A breakthrough was then made by Hoffman \cite{hoffman2008geodesics}, where he used Busemann functions to study infinite geodesics.
These techniques then led to more progress in the geometric structure of geodesics, see e.g. \cite{damron2014busemann, damron2017bigeodesics}.

For exactly solvable DLPP models, 
there are also several motivations to study the geometry of geodesics (which are maximal weight paths under this setting).
For example, for models obtained by adding local defects to exactly solvable ones, the integrable structures are destroyed, thus the geometric properties of the geodesics play an important role in the study of such models; see e.g. \cite{basu2014last}, where the authors settled the ``slow bond problem'', where extra weights are added to the diagonal in the exactly solvable DLPP models.
Besides,
the coalescence of geodesics in exactly solvable models would help to understand the geometry of the scaling limiting objects of the KPZ universality class, e.g. the coalescence structure in Brownian LPP is used towards understanding Brownian regularity of the Airy process \cite{hammond2019patchwork, hammond2020exponents}.

In recent years there are several results about geodesics in DLPP models, see e.g. \cite{coupier2011multiple, ferrari2005competition};
and some of these results have been proved beyond exactly solvable models, see \cite{georgiou2017geodesics, georgiou2017stationary}.
There is the problem of the speed of coalescence, i.e. to understand the distribution of the coalescence location of two geodesics.
For two semi-infinite geodesics in the same direction, such problem was first studied in \cite{wuthrich2002asymptotic}.
A lower bound of the tail of the distribution was obtained by Pimentel in \cite{pimentel2016duality}, using a duality argument.
A corresponding upper bound was conjectured in \cite{pimentel2016duality}, and
this was settled by Basu, Sarkar, and Sly in \cite{basu2017coalescence}.
In a different line of works, using inputs from the connection of the stationary LPP with queuing theory, Sepp{\"a}l{\"a}inen and Shen (\cite{seppalainen2020coalescence}) also obtained an upper bound which falls short of the optimal order by a logarithmic factor. They also obtained upper and lower bounds of matching orders for the probability of the coalescence location being very close to the starting points.
More recently, in \cite{busani2020universality} various coalescence results are obtained, and are used to deduce universality of geodesic trees.

In this note, we study coalescence of finite geodesics, in the model of DLPP with exponential weights.
We consider two geodesics, from two distinct points to the same finite point in the $(1, 1)$ direction, and study
the tail of the distribution of their coalescence location.
This problem was also studied in \cite{basu2017coalescence}, where the tail of the distribution of the coalescence location was conjectured to have the same order as that of semi-infinite geodesics, and the authors also gave an upper bound of polynomial decay.
We settle this problem in this note, by providing matching upper and lower bounds up to a constant factor.
Very soon since this note was posted, similar result was also obtained in \cite{balazs2020local}, with slightly more restricted conditions (using notations of Theorem \ref{thm:main} below, in \cite{balazs2020local} the estimates are proven for $n>(Rk)^5$, while we only assume that $n>Rk$).

As \cite{basu2017coalescence}, our proofs work, essentially verbatim, for several generalizations; while we state and prove our results in the current form for technical and notational convenience.
First, following our arguments, the same result also holds when consider two finite geodesics in any fixed direction (rather than the $(1, 1)$ direction), except the axial directions.
Besides, our arguments generalize to some other exactly solvable models beyond DLPP with exponential weights.
These include Poissonian DLPP in continuous space $\R^2$, and DLPP with geometric weights.
The difference is that, unlike DLPP with exponential weights, under these two settings the geodesics are not almost surely unique. Thus we consider the right-most (or left-most) geodesics between pairs of points, and study their coalescence location instead. The same estimates as our main result (Theorem \ref{thm:main} below) can be obtained.
For these two settings, we would also need the Tracy-Widom limit and one point upper and lower tail moderate deviation estimates for the last passage times (as given by \cite{johansson2000shape, LR10} for exponential DLPP).
For Poissionian DLPP these can be found in \cite{lowe2001moderate, lowe2002moderate}, and for geometric DLPP these are in \cite{26f0ff4cdf2a4bd6aa9281654a909d53, corwin2016fluctuations}.

\subsection{Notations, main results, and proof ideas}
We set up notations for the model and formally state our results here.
Consider the 2D lattice $\Z^2$.
For each $v \in \Z^2$ we associate a weight $\xi_{v}$, which is distributed as $\Exp(1)$ and independent from each other.
For any upper-right oriented path $\gamma$ in $\Z^2$, we define the \emph{passage time of the path} to be
\[
T(\gamma) := \sum_{v \in \gamma} \xi_v .
\]
For any $u, v \in \Z^2$, we denote $u\leq v$, if $u$ is less or equal to $v$ in each coordinate;
and we denote $u<v$ if $u\leq v$ and $u\neq v$.
For any $u < v \in \Z^2$, there are finitely many upper-right paths from $u$ to $v$.
Almost surely, there is a unique one $\gamma$ with the largest $T(\gamma)$.
We denote it to be the \emph{geodesic} $\Gamma_{u,v}$, and $T_{u,v}:=T(\gamma)$ to be the \emph{passage time from $u$ to $v$}.
For any $u=(u_1, u_2) \in \Z^2$, we denote $d(u):=u_1+u_2$.
For each $n \in \Z$, we denote $\LL_{n}$ to be the line $\{v \in \Z^2: d(v)= n \}$.
For any $u, v < w \in \Z^2$, we denote $\fC^{u, v; w}=(\fC^{u, v; w}_1, \fC^{u, v; w}_2) \in \Z^2$ to be the \emph{first coalescence point} of $\Gamma_{u, w}$ and $\Gamma_{v, w}$, i.e. $\fC^{u, v; w} \in \Gamma_{u, w}\cap \Gamma_{v, w}$ with the smallest $d(\fC^{u, v; w})$.

For any $n, k \in \Z_+$, we denote $\bn:=(n,n)$ and $\bbk:=(\lfloor k^{2/3} \rfloor, -\lfloor k^{2/3} \rfloor)$. In particular, we let $\boo:=(0,0)$.
Our result is about the location of $\fC^{\bbk, -\bbk; \bn}$.
\begin{theorem}   \label{thm:main}
There exists universal constants $C_1, C_2, R_0 > 0$, such that for any $k, n \in \Z_+$, $R > R_0$, and $n > Rk$, we have
\[
C_1R^{-2/3} < \PP[ d(\fC^{\bbk, -\bbk; \bn} )> Rk ] < C_2R^{-2/3}.\]
\end{theorem}
Denote $\hbk:=(0, \lfloor k^{2/3} \rfloor)$.
In \cite[Theorem 1]{basu2017coalescence} the problem was presented in a slightly different setting,
where the first coordinate of $\fC^{\boo, \hbk; \bn}$ was studied; and it was shown there that $\PP[\fC^{\boo, \hbk; \bn}_1>Rk]<CR^{-c}$, for some constants $C, c > 0$.
Our result confirms that the optimal $c$ is $\frac{2}{3}$.
\begin{corollary}\label{cor:bsssetting}
There exists universal constants $C_1', C_2', R_0 > 0$, such that for any $k, n \in \Z_+$, $R > R_0$, and $n > 4Rk$, we have
\[
C_1'R^{-2/3} < \PP[\fC^{\boo, \hbk; \bn}_1>Rk] < C_2'R^{-2/3}.
\]
\end{corollary}

Let us explain the main ideas and difficulties in proving these results, and the new ingredients based on previous works.
To get matching upper and lower bounds for the coalescence probability $\PP[d(\fC^{\bbk, -\bbk; \bn})>Rk]$, which is a small quantity, a general idea from \cite{basu2017coalescence} is to ``magnify'' it, i.e. to prove that multiplied by $R^{2/3}$ it is upper and lower bounded by constants.
For semi-infinite geodesics in the same direction, all of them have the same distribution (under translations).
The argument in \cite{basu2017coalescence} is to 
take $\lfloor CR^{2/3} \rfloor+1$ points in $\LL_0$ of equal distances (for some constant $C$), and take the $(1,1)$-direction semi-infinite geodesic from each of them.
For any two ``neighboring'' semi-infinite geodesics, the probability that they do not coalesce before $\LL_{\lfloor Rk \rfloor}$ is the same by the translation invariance.
Then this probability multiplied by $\lfloor CR^{2/3} \rfloor$ is precisely the expected number of different intersecting points of these $\lfloor CR^{2/3} \rfloor+1$ semi-infinite geodesics with $\LL_{\lfloor Rk \rfloor}$, minus $1$.
This quantity is shown to be in constant order, using results from \cite{basu2018nonexistence, basu2014last}.

However, to study coalescence of finite geodesics ending at the same point, one could not directly magnify the probability in a similar way. The main difference is that finite geodesics (from different points to the same end point) do not have the same distribution.
Thus approaches from some different directions were taken to study coalescence of finite geodesics.
In \cite{basu2017coalescence}, a multi-scale argument is used; and the bounds in \cite{balazs2020local} are obtained from a comparison between finite and semi-infinite geodesics.
Our main contribution is a short geometric construction utilizing invariances of this model in a novel way, which allows us to still magnify the coalescence probability by $R^{2/3}$, and to get matching bounds with the least requirements.
Specifically, in addition to translation invariance, a key and simple observation is that the model is also invariant under rotation by $\pi$.
With this, we convert the study of coalescence of $\Gamma_{-\bbk, \bn}$ and $\Gamma_{\bbk, \bn}$ to that of coalescence of $\Gamma_{-\bbk, \bn}$ and $\Gamma_{\bbk, \bn+2\bbk}$.
The advantage of the later one is that $\Gamma_{\bbk, \bn+2\bbk}$ has the same law as $\Gamma_{-\bbk, \bn}$, after translating by $2\bbk$ (and we call them \emph{parallel} geodesics).
Then to do the magnification, we make $\lfloor CR^{2/3}\rfloor$ translations by $2\bbk$ each time, and consider a family of $\lfloor CR^{2/3}\rfloor+1$ parallel geodesics.
It remains to study the number of their intersections with $\LL_{\lfloor Rk \rfloor}$, and for this we adapt arguments from \cite{basu2017coalescence} and also use results from \cite{basu2018nonexistence, basu2014last}.

\section{Geometric construction and parallel geodesics}
\subsection{Preliminaries}
We start with some basic geometric properties of geodesics.
A first observation is that, for any $u\leq v\in\Z^2$, if $u'\leq v'$ and $u',v'\in\Gamma_{u,v}$, we must have that $\Gamma_{u',v'}\subset \Gamma_{u,v}$, and $\Gamma_{u',v'}$ is the part of $\Gamma_{u,v}$ between $u'$ and $v'$ (including $u', v'$).
This immediately leads to the following result.
\begin{lemma}  \label{l:intersec}
Take points $u<v$ and $u'<v'$.
Then $\Gamma_{u,v}\cap\Gamma_{u',v'}$ is either empty, or equals $\Gamma_{u'',v''}$ for some $u''\leq v''$.
\end{lemma}
In particular, this implies that for $u,v<w\in\Z^2$, the geodesics $\Gamma_{u,w}$ and $\Gamma_{v,w}$ follow the same path from $\fC^{u, v; w}$ to $w$.

We define another (partial) order of points in $\Z^2$:
for $u=(u_1,u_2), v=(v_1,v_2)\in\Z^2$, we denote $u\preceq v$, if $u_1\leq v_1$ and $u_2\geq v_2$, and we denote $u\prec v$ if $u\preceq v$ and $u\neq v$.
From Lemma \ref{l:intersec} we get the following result which says that geodesics are ordered.
\begin{lemma}  \label{l:ordered}
Suppose $u, u', v, v' \in \Z^2$ satisfy that $u\leq v$, $u'\leq v'$, and $u\preceq u'$, $v\preceq v'$.
Then for any $w\in\Gamma_{u,v}$ and $w'\in\Gamma_{u',v'}$, we cannot have $w'\prec w$.
\end{lemma}
The next (less elementary) result is an estimate about spatial transversal fluctuation of geodesics, and will be repeatedly used in the rest of this note.
It can be thought of as \cite[Theorem 3]{basu2017coalescence} in a slightly different form.
\begin{proposition}  \label{prop:spaflu}
There are absolute constants $c, r_0 \in \R_+$ such that the following is true. 
Let $n,r\in\Z_+$ with $r_0<r<n$, and $m \in \Z$ with $|m|<10r^{2/3}$.
Take $f_0 \in \Z$ such that $(r+f_0, r-f_0) = \Gamma_{(m,-m), \bn} \cap \LL_{2r}$, then we have $\PP[|f_0| > xr^{2/3}] < e^{-cx}$ for any large enough $x$.
\end{proposition}
The proof of this proposition is similar to that of \cite[Theorem 3]{basu2017coalescence}, 
and the arguments could be traced back to \cite{newman1995surface} in the setting of first passage percolation.
We leave the proof to the appendix.

\subsection{Proof of the main results: rotation and translation}
Exploiting rotation invariance of this model, we convert the probability considered in Theorem \ref{thm:main} to the probability of coalescence of two parallel geodesics.
\begin{lemma}   \label{lem:redtoparall}
Take any $k, n\in \Z_+$, $R > 10$, with $n > Rk$.
We have
\begin{equation}  \label{eq:redtoparall}
\frac{1}{2}
\leq \frac{\PP[ d(\fC^{\bbk, -\bbk; \bn} )> Rk]}{ 
\PP [ \Gamma_{-\bbk, \bn} \cap \LL_{\lfloor Rk \rfloor} \neq \Gamma_{\bbk, \bn + 2\bbk } \cap \LL_{\lfloor Rk \rfloor} ]}
\leq 1
\end{equation}
\end{lemma}
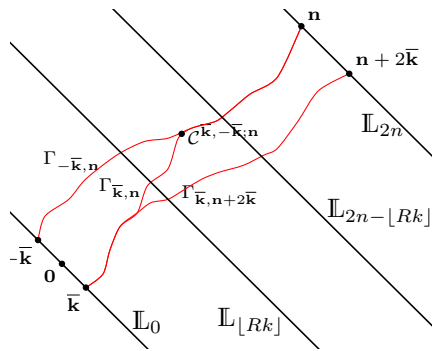
\begin{figure}[htbp]
    \centering
\begin{tikzpicture}[line cap=round,line join=round,>=triangle 45,x=4.5cm,y=4.5cm]
\clip(0.2,0.1) rectangle (1.43,1.1);
\begin{scriptsize}
\draw (0.35,0.35) node[anchor=north east]{$\boo$};
\draw (0.42,0.28) node[anchor=north east]{$\bbk$};
\draw (0.28,0.42) node[anchor=north east]{$-\bbk$};
\draw (0.7,0.733) node[anchor=west]{$\fC^{\bbk,-\bbk;\bn}$};
\draw (0.6,0.62) node[anchor=north east]{$\Gamma_{\bbk,\bn}$};
\draw (0.48,0.7) node[anchor=north east]{$\Gamma_{-\bbk,\bn}$};
\draw (0.94,0.6) node[anchor=north east]{$\Gamma_{\bbk,\bn+2\bbk}$};
\draw (1.05,1.05) node[anchor=south west]{$\bn$};
\draw (1.19,0.91) node[anchor=south west]{$\bn+2\bbk$};
\end{scriptsize}
\draw (0.53,0.18) node[anchor=west]{$\LL_0$};
\draw (0.78,0.18) node[anchor=west]{$\LL_{\lfloor Rk \rfloor}$};
\draw (1.2,0.76) node[anchor=west]{$\LL_{2n}$};
\draw (1.1,0.5) node[anchor=west]{$\LL_{2n-\lfloor Rk \rfloor}$};

\draw [red] plot [smooth] coordinates {(1.05,1.05) (0.98,0.90) (0.9,0.85) (0.84,0.8) (0.8,0.77) (0.73,0.75) (0.66,0.72) (0.59,0.71) (0.51,0.67) (0.42, 0.6) (0.37,0.54) (0.3,0.49) (0.28,0.42)};
\draw [red] plot [smooth] coordinates {(1.05,1.05) (0.98,0.90) (0.9,0.85) (0.84,0.8) (0.8,0.77) (0.73,0.75) (0.69, 0.72) (0.66,0.63) (0.6,0.58) (0.57,0.50) (0.52, 0.46) (0.49,0.39) (0.47,0.32) (0.42,0.28)};
\draw [red] plot [smooth] coordinates {(1.19,0.91) (1.1,0.86) (1.04,0.77) (0.98,0.69) (0.94,0.67) (0.88,0.64) (0.81, 0.62) (0.76,0.6) (0.71,0.57) (0.66,0.54) (0.61,0.53) (0.57,0.50) (0.52, 0.46) (0.49,0.39) (0.47,0.32) (0.42,0.28)};
\draw [line width=.6pt] (-0.05,0.75) -- (0.75,-0.05);
\draw [line width=.6pt] (0.0,2.1) -- (2.1,0.0);
\draw [line width=.6pt] (0.0,1.2) -- (1.2,0.0);
\draw [line width=.6pt] (0.0,1.6) -- (1.6,0.0);
\draw [fill=black] (0.35,0.35) circle (1.0pt);
\draw [fill=black] (0.28,0.42) circle (1.0pt);
\draw [fill=black] (0.42,0.28) circle (1.0pt);
\draw [fill=black] (1.05,1.05) circle (1.0pt);
\draw [fill=black] (1.19,0.91) circle (1.0pt);
\draw [fill=black] (0.7,0.733) circle (1.0pt);
\end{tikzpicture}
\caption{An illustration of the objects in the proof of Lemma \ref{lem:redtoparall}.}
\end{figure}
\begin{proof}
Denote $\cE_1$ to be the event where
$\Gamma_{\bbk, \bn} \cap \LL_{\lfloor Rk \rfloor} \neq \Gamma_{-\bbk, \bn} \cap \LL_{\lfloor Rk \rfloor}$.
Then we note that the event $d(\fC^{\bbk, -\bbk; \bn} )> Rk$ is equivalent to $\cE_1$, by Lemma \ref{l:intersec}.
Also, denote $\cE_2$ to be the event where
\[
\Gamma_{\bbk, \bn} \cap \LL_{\lfloor Rk \rfloor} \neq \Gamma_{\bbk, \bn + 2\bbk} \cap \LL_{\lfloor Rk \rfloor},
\]
and $\cE_3$ to be the event where
\[
\Gamma_{-\bbk, \bn} \cap \LL_{\lfloor Rk \rfloor} \neq \Gamma_{\bbk, \bn + 2\bbk} \cap \LL_{\lfloor Rk \rfloor} .
\]
Now \eqref{eq:redtoparall} is equivalent to $\frac{1}{2}\PP[\cE_3]\leq \PP[\cE_1] \leq \PP[\cE_3]$.

We define $\cE_1'$ as $\cE_1$ rotated by $\pi$ around $\frac{\bbk+\bn}{2}$, i.e.
\[
\Gamma_{\bbk, \bn} \cap \LL_{2n - \lfloor Rk \rfloor} \neq \Gamma_{\bbk, \bn + 2\bbk} \cap \LL_{2n - \lfloor Rk \rfloor}.
\]
We have $\PP[\cE_1']=\PP[\cE_1]$ by rotation invariance of this model.
Note that both $\cE_2$ and $\cE_1'$ are about the geodesics $\Gamma_{\bbk, \bn}, \Gamma_{\bbk, \bn + 2\bbk}$.
Since $n > Rk$, we have $2n - \lfloor Rk \rfloor > \lfloor Rk \rfloor$, so $\cE_2$ implies $\cE_1'$ by Lemma  \ref{l:intersec}, and $\PP[\cE_2]\leq \PP[\cE_1']$.
By Lemma \ref{l:ordered} we have
$\Gamma_{-\bbk, \bn} \cap \LL_{\lfloor Rk \rfloor} \preceq \Gamma_{\bbk, \bn} \cap \LL_{\lfloor Rk \rfloor} \preceq \Gamma_{\bbk, \bn + 2\bbk} \cap \LL_{\lfloor Rk \rfloor}$, so $\cE_3=\cE_1 \cup \cE_2$.
Thus
\[
\PP[\cE_3] \leq \PP[\cE_1] + \PP[\cE_2] \leq \PP[\cE_1] + \PP[\cE_1'] = 2\PP[\cE_1] \leq 2\PP[\cE_3],
\]
and our conclusion follows.
\end{proof}
Our next step is to study coalescence of two parallel geodesics.
For this, we consider a family of parallel geodesics as follows.

Take $a, b \in \Z$ and $m, d, s \in \Z_+$.
We define two sequences of points: $u_0, u_1, \cdots, u_m \in \LL_0$ and $v_0, v_1, \cdots, v_m \in \LL_s$, where $u_i := (a+id, -a-id)$, and $v_i := (\lfloor s/2\rfloor+b+id, \lceil s/2\rceil-b-id)$
for each $1 \leq i \leq m$.
We take the family of geodesics $\{\Gamma_{u_i, v_i}\}_{i=1}^m$, and study the number of intersections of them with $\LL_r$, for some $0 < r < s$.
We will show that when $md$ is in the order of $r^{2/3}$, its expectation is lower and upper bounded by constants.
\begin{proposition}   \label{prop:lowerbound}
There exist constants $M, r_0 \in \R_+$, such that if $md > Mr^{2/3}$, $r_0 < r < s$, and $|a|, |b| < r^{2/3}$, we have
\begin{equation}
\E | \LL_r \cap \cup_{i=1}^m \Gamma_{u_i, v_i} |  > \frac{3}{2},
\end{equation}
\end{proposition} 
\begin{proposition}  \label{prop:upperbound}
There exist constants $C, r_0 \in \R_+$, such that if $r > r_0$, $s > \frac{3r}{2}$, and $md, |a|, |b| < r^{2/3}$, we have
\begin{equation}
\E | \LL_r \cap \cup_{i=0}^m \Gamma_{u_i, v_i} |  < C.
\end{equation}
\end{proposition}
The proofs of these results are adapted from the proof \cite[Theorem 2]{basu2017coalescence}, and we leave them to the next subsection.
We now prove our main results assuming these two propositions.
\begin{proof}[Proof of Theorem \ref{thm:main}]
For the parameters in the setting of Proposition \ref{prop:lowerbound} and \ref{prop:upperbound}, we take $r = \lfloor Rk \rfloor$, $s = 2n$, $a = -\lfloor k^{2/3}\rfloor$, $b = 0$, $d = 2\lfloor k^{2/3}\rfloor$.
We leave $m \in \Z_+$ to be determined.
Then we have $u_0=-\bbk$, $u_1=\bbk$, $v_0=\bn$, and $v_1=\bn+2\bbk$.
In view of Lemma \ref{lem:redtoparall}, we just need to upper and lower bound $\PP[\Gamma_{u_0, v_0} \cap \LL_r \neq \Gamma_{u_1, v_1} \cap \LL_r ]$.

By letting $R$ large, we can make $r$ large, $|a|, |b| < r^{2/3}$, and $s > \frac{3r}{2}$ (since we require that $n > Rk$). 
By translation invariance, for each $0 \leq i \leq m-1$, we have
\[
\PP[\Gamma_{u_0, v_0} \cap \LL_r \neq \Gamma_{u_1, v_1} \cap \LL_r ] = \PP[\Gamma_{u_i, v_i} \cap \LL_r \neq \Gamma_{u_{i+1}, v_{i+1}} \cap \LL_r].
\]
This (with Lemma \ref{l:ordered}) implies that
\[
\begin{split}
\E | \LL_r\cap \cup_{i=0}^m \Gamma_{u_i, v_i} | & =
1 + \sum_{i=0}^{m-1}\PP[\Gamma_{u_i, v_i} \cap \LL_r \neq \Gamma_{u_{i+1}, v_{i+1}} \cap \LL_r]
\\
&=
1 +
m\PP[\Gamma_{u_0, v_0} \cap \LL_r \neq \Gamma_{u_1, v_1} \cap \LL_r ].
\end{split}
\]
First, we take $m = \lceil MR^{2/3} \rceil$, where $M$ is from Proposition \ref{prop:lowerbound}.
By taking $R$ large enough, we have $m < 2MR^{2/3}$, and $md \geq 2MR^{2/3}\lfloor k^{2/3}\rfloor > Mr^{2/3}$. Hence by Proposition \ref{prop:lowerbound} we have
\[
\PP[\Gamma_{u_0, v_0} \cap \LL_r \neq \Gamma_{u_1, v_1} \cap \LL_r ] > \frac{1}{2m} > (4M)^{-1}R^{-2/3}.
\]
Second, we take $m = \lfloor R^{2/3}/3 \rfloor$. 
Now we have $m > R^{2/3}/4$ and $md \leq 2(Rk)^{2/3}/3 < r^{2/3}$ when $R$ is large enough.
By Proposition \ref{prop:upperbound}, we have that
\[
\PP[\Gamma_{u_0, v_0} \cap \LL_r \neq \Gamma_{u_1, v_1} \cap \LL_r ] < \frac{C}{m} < 4CR^{-2/3}.
\]
Finally, putting these together and using Lemma \ref{lem:redtoparall}, we get
\[
(8M)^{-1}R^{-2/3} <
\PP[ d(\fC^{\bk^{(1)}, \bk^{(2)}; \bn} ) > Rk  ] 
 < 4CR^{-2/3},
\]
and our conclusion follows since $M$ and $C$ are constants.
\end{proof}
\begin{proof}[Proof of Corollary \ref{cor:bsssetting}]
For the upper bound, first note that $\fC^{\boo, \hbk; \bn} \leq \fC^{\bbk, -\bbk; \bn}$, which is due to $-\bbk\prec \boo, \hbk \prec \bbk$ and Lemma \ref{l:intersec}, \ref{l:ordered}.
Then applying Theorem \ref{thm:main} we have
\[
\PP[ \fC^{\boo, \hbk; \bn}_1> Rk]
\leq
\PP[d(\fC^{\boo, \hbk; \bn})> Rk ]
\leq
\PP[ d(\fC^{\bbk, -\bbk; \bn})> Rk ]
< C_2R^{-2/3}.
\]
Now we prove the lower bound.
Denote $\hbk':=(\lfloor k^{2/3} \rfloor, 0)$.
By symmetry of the model we have
\[
2\PP[ d(\fC^{\boo, \hbk; \bn})> 4Rk ] =
\PP[ d(\fC^{\boo, \hbk; \bn})> 4Rk ]
+
\PP[ d(\fC^{\boo, \hbk'; \bn})> 4Rk ]
\geq
\PP[ d(\fC^{\hbk, \hbk'; \bn})> 4Rk ],
\]
where for the second inequality we use that $d(\fC^{\hbk, \hbk'; \bn}) = \max \{d(\fC^{\boo, \hbk; \bn}), d(\fC^{\boo, \hbk'; \bn})\}$, which is due to Lemma \ref{l:intersec}.
By Theorem \ref{thm:main}, we have that
\[
\PP[ d(\fC^{\hbk, \hbk'; \bn})> 4Rk ] > C_1R^{-2/3}/8.
\]
Next, by Proposition \ref{prop:spaflu} we have
\[
\PP[ d(\fC^{\boo, \hbk; \bn})> 4Rk, \fC^{\boo, \hbk; \bn}_1\leq Rk ]
\leq
\PP[f_{0} < -Rk] < e^{-c R^{2/3}k^{2/3}},
\]
where $f_{0} \in \Z$ such that $(\lceil 2Rk\rceil + f_0, \lceil 2Rk\rceil - f_0) \in \Gamma_{\boo, \bn}$, and $c>0$ is an absolute constant.
Thus we have
\[
\begin{split}
\PP[ \fC^{\boo, \hbk; \bn}_1> Rk ] &
\geq
\PP[ d(\fC^{\boo, \hbk; \bn} )> 4Rk ]
-
\PP[ d(\fC^{\boo, \hbk; \bn})> 4Rk, \fC^{\boo, \hbk; \bn}_1\leq Rk ]
\\ &
\geq C_1R^{-2/3}/16 - e^{-c R^{2/3}k^{2/3}}.
\end{split}
\]
Then by taking $R_0$ large enough the lower bound follows.
\end{proof}

\subsection{Intersections with parallel geodesics}
In this subsection we prove Proposition \ref{prop:lowerbound} and \ref{prop:upperbound}, adapting the arguments in \cite[Section VI]{basu2017coalescence}.
We start with the lower bound, which follows from the spatial transversal fluctuation (Proposition \ref{prop:spaflu}).
\begin{proof}[Proof of Proposition \ref{prop:lowerbound}]
It suffices to consider two geodesics $\Gamma_{u_0, v_0}$ and $\Gamma_{u_m, v_m}$.

Take $f_0, f_m \in \Z$ such that $(\lfloor \frac{r}{2}\rfloor+f_0, \lceil \frac{r}{2}\rceil-f_0) = \Gamma_{u_0, v_0} \cap \LL_r$, and $(\lfloor \frac{r}{2}\rfloor+f_m, \lceil \frac{r}{2}\rceil-f_m) = \Gamma_{u_r, v_r} \cap \LL_r$.
By Proposition \ref{prop:spaflu}, and translation invariance, we can find constant $p\in\R_+$ such that 
\[
\PP[|f_0|>pr^{2/3}] ,\; \PP[|f_m - md|>pr^{2/3}] < \frac{1}{4}.
\]
By taking $M > 2p$, we have $md > Mr^{2/3} > 2pr^{2/3}$, and $\PP[f_0 = f_m] < \frac{1}{2}$.
Then
\[
\E| \LL_r \cap \cup_{i=1}^m \Gamma_{u_i, v_i} |  \geq \E| \LL_r \cap \left(\Gamma_{u_0, v_0} \cup \Gamma_{u_m, v_m} \right) |
\geq
1 + \PP[f_0\neq f_m] > \frac{3}{2},
\]
and our conclusion follows.
\end{proof}
\begin{figure}[htbp]
    \centering
\begin{tikzpicture}[line cap=round,line join=round,>=triangle 45,x=3cm,y=5.5cm]
\clip(-1.2,-0.1) rectangle (2.0,0.55);
\begin{scriptsize}
\draw (-0.2,0.0) node[anchor=north]{$u_0$};
\draw (0.0,0.0) node[anchor=north]{$u_1$};
\draw (0.2,0.0) node[anchor=north]{$u_2$};
\draw (0.4,0.0) node[anchor=north]{$u_3$};
\draw (0.6,0.0) node[anchor=north]{$u_4$};
\draw (0.8,0.0) node[anchor=north]{$u_5$};
\draw (1.0,0.0) node[anchor=north]{$u_6$};
\draw (-0.2,0.5) node[anchor=south]{$v_0$};
\draw (0.0,0.5) node[anchor=south]{$v_1$};
\draw (0.2,0.5) node[anchor=south]{$v_2$};
\draw (0.4,0.5) node[anchor=south]{$v_3$};
\draw (0.6,0.5) node[anchor=south]{$v_4$};
\draw (0.8,0.5) node[anchor=south]{$v_5$};
\draw (1.0,0.5) node[anchor=south]{$v_6$};
\draw (-0.15,0.17) node[anchor=south east]{$w_0$};
\draw (0.0244,0.17) node[anchor=south]{$w_1/w_2$};
\draw (0.383,0.17) node[anchor=south]{$w_3$};
\draw (0.436,0.17) node[anchor=south west]{$w_4$};
\draw (0.952,0.17) node[anchor=south west]{$w_5/w_6$};

\end{scriptsize}
\draw (1.8,0.0) node[anchor=north]{$\LL_0$};
\draw (1.8,0.17) node[anchor=north]{$\LL_{r}$};
\draw (1.8,0.5) node[anchor=north]{$\LL_{s}$};

\draw [red] plot [smooth] coordinates {(1.0,0.0) (0.96,0.05) (0.8,0.1) (0.92,0.15) (0.96,0.2) (0.77,0.27) (0.79,0.33) (0.85,0.37) (1.03,0.43)
(1.0,0.5)};
\draw [red] plot [smooth] coordinates {(0.8,0.0) (0.76,0.05) (0.78,0.1) (0.92,0.15) (0.96,0.2) (0.77,0.27) (0.7,0.35) (0.73,0.43)
(0.8,0.5)};
\draw [red] plot [smooth] coordinates {(0.6,0.0) (0.62,0.05) (0.57,0.1) (0.43,0.15) (0.46,0.2) (0.52,0.3) (0.67,0.35) (0.72,0.43)
(0.6,0.5)};
\draw [red] plot [smooth] coordinates {(0.4,0.0) (0.29,0.05) (0.25,0.1) (0.33,0.15) (0.45,0.2) (0.51,0.3) (0.44,0.35) (0.38,0.43)
(0.4,0.5)};
\draw [red] plot [smooth] coordinates {(0.2,0.0)
(0.25,0.1) (0.03,0.15) (0.03,0.2) (0.02,0.25) (0.11,0.3) (0.22,0.35) (0.27,0.43)
(0.2,0.5)};
\draw [red] plot [smooth] coordinates {(0.0,0.0)
(-0.05,0.1) (0.03,0.2) (-0.11,0.3) (-0.09,0.35) (0.04,0.43)
(0.0,0.5)};
\draw [red] plot [smooth] coordinates {(-0.2,0.0)
(-0.24,0.1) (-0.12,0.2) (-0.11,0.3) (-0.27,0.4)
(-0.2,0.5)};
\draw [line width=.6pt] (-10,0.0) -- (10,0.0);
\draw [line width=.6pt] (-10,0.5) -- (10,0.5);
\draw [line width=.6pt] (-10,0.17) -- (10,0.17);
\draw [fill=black] (-0.2,0.0) circle (1.0pt);
\draw [fill=black] (0.0,0.0) circle (1.0pt);
\draw [fill=black] (0.2,0.0) circle (1.0pt);
\draw [fill=black] (0.4,0.0) circle (1.0pt);
\draw [fill=black] (0.6,0.0) circle (1.0pt);
\draw [fill=black] (0.8,0.0) circle (1.0pt);
\draw [fill=black] (1.0,0.0) circle (1.0pt);
\draw [fill=black] (-0.2,0.5) circle (1.0pt);
\draw [fill=black] (0.0,0.5) circle (1.0pt);
\draw [fill=black] (0.2,0.5) circle (1.0pt);
\draw [fill=black] (0.4,0.5) circle (1.0pt);
\draw [fill=black] (0.6,0.5) circle (1.0pt);
\draw [fill=black] (0.8,0.5) circle (1.0pt);
\draw [fill=black] (1.0,0.5) circle (1.0pt);
\draw [fill=black] (-0.15,0.17) circle (1.0pt);
\draw [fill=black] (0.0244,0.17) circle (1.0pt);
\draw [fill=black] (0.383,0.17) circle (1.0pt);
\draw [fill=black] (0.436,0.17) circle (1.0pt);
\draw [fill=black] (0.952,0.17) circle (1.0pt);
\end{tikzpicture}
\caption{An illustration of the objects in the proof of Proposition \ref{prop:upperbound}, rotated by $\pi/4$.} \label{f:fig2}
\end{figure}
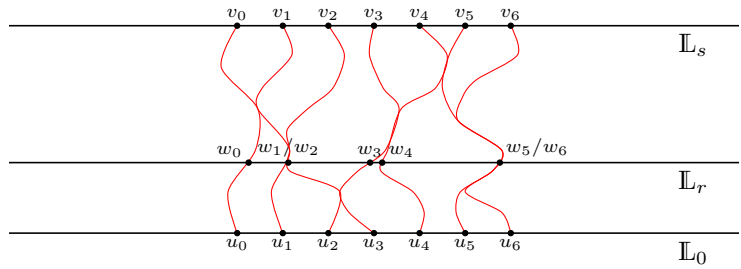
For the upper bound, in addition to Proposition \ref{prop:spaflu} we also need the following estimate on the number of disjoint geodesics.
\begin{proposition}[\protect{\cite[Corollary 2.7]{basu2018nonexistence}}] \label{prop:numnoncoalgeo}
There exists $n_0, \ell_0,c \in \R_+$ such that the following is true.
Take any $n, \ell \in \Z_+$, $n>n_0$, $n^{0.01}>\ell > \ell_0$.
Let $A_{\ell,n}:= \{ (i, -i): i \in \Z, |i| < \ell^{1/16} n^{2/3} \} \subset \LL_0$, and $B_{\ell,n}:= \{ (n+i, n-i): i \in \Z, |i| < \ell^{1/16} n^{2/3} \} \subset \LL_{2n}$.
Let $\cC_{\ell,n}$ be the event that there exists $\tu_1, \cdots, \tu_{\ell} \in A_{\ell,n}$ and $\tv_1, \cdots, \tv_{\ell} \in B_{\ell,n}$, such that the geodesics $\{\Gamma_{\tu_i, \tv_i}\}_{i=1}^{\ell}$ are mutually disjoint.
Then $\PP[\cC_{\ell,n}] < e^{-c\ell^{1/8}}$.
\end{proposition}
In \cite{basu2018nonexistence}, this result is stated as a generalization of \cite[Theorem 2]{basu2018nonexistence}
(see also \cite[Theorem 3.2]{BGHH} for a refinement of the arguments).
For completeness we also contain a proof of this proposition in the appendix.

Now we finish the proof of the upper bound.
\begin{proof}[Proof of Proposition \ref{prop:upperbound}]
We can assume that $r>10^{1001}$, since otherwise the result follows by taking $C$ large enough.

Take an absolute constant $\tau_0 := \max\{10^{10}, 2\ell_0+1\}$, where $\ell_0$ is from Proposition \ref{prop:numnoncoalgeo}.
For any $\tau \in \Z_+$, $\tau_0<\tau < r^{0.001}$, we wish to bound $\PP[| \LL_r \cap \cup_{i=0}^m \Gamma_{u_i, v_i}| > \tau ]$.
We denote $w_i := \Gamma_{u_i, v_i} \cap \LL_r$ for each $i=0,\cdots, m$.
Then by Lemma \ref{l:ordered} we have $w_0\preceq \cdots \preceq w_m$.
For any $0 \leq i < j \leq m$, if $w_i \neq w_j$, then by Lemma \ref{l:intersec}
at least one of $\Gamma_{u_i, w_i} \cap \Gamma_{u_j, w_j} = \emptyset$ and $\Gamma_{w_i, v_i} \cap \Gamma_{w_j, v_j} = \emptyset$ happens (see Figure \ref{f:fig2} for an illustration).
Now we denote
\[
\begin{aligned}
I:&=\{i \in \{0, \cdots, m-1\}: w_i \neq w_{i+1}\}, \\
I_1:&=\{i \in \{0, \cdots, m-1\}: \Gamma_{u_i, w_i} \cap \Gamma_{u_{i+1}, w_{i+1}} = \emptyset \}, \\
I_2:&=\{i \in \{0, \cdots, m-1\}: \Gamma_{w_i, v_i} \cap \Gamma_{w_{i+1}, v_{i+1}} = \emptyset \}.
\end{aligned}
\]
From this definition we have that $I = I_1 \cup I_2$. We also have that $\{\Gamma_{u_i, w_i}\}_{i\in I_1}$ are mutually disjoint, and that $\{\Gamma_{w_i, v_i}\}_{i\in I_2}$ are mutually disjoint.

We let $f_0, f_m \in \Z$ such that $(\lfloor \frac{r}{2} \rfloor+f_0, \lceil \frac{r}{2} \rceil-f_0) = w_0$, and $(\lfloor \frac{r}{2} \rfloor+f_m, \lceil \frac{r}{2} \rceil-f_m) = w_m$.
Now suppose that $| \LL_r \cap \cup_{i=0}^m \Gamma_{u_i, v_i}| > \tau $, and $|f_0|, |f_m|<\lfloor \tau/2\rfloor^{1/16} \lfloor r/4\rfloor^{2/3}-1$.
Then $|I| \geq \tau$, so either $|I_1| \geq \tau/2$ or $|I_2| \geq \tau/2$.
Using the notations in Proposition \ref{prop:numnoncoalgeo}, 
these imply that either the event $\cC_{\lfloor \tau/2\rfloor, \lfloor r/2 \rfloor}$ happens, or the event $\cC_{\lfloor \tau/2\rfloor, \lfloor (s-r)/2 \rfloor}$ translated by $(\lceil r/2 \rceil, \lceil r/2 \rceil)$ happens (here we use that $\lfloor (s-r)/2 \rfloor \geq \lfloor r/4\rfloor$ since $s > 3r/2$). Thus $\PP[| \LL_r \cap \cup_{i=0}^m \Gamma_{u_i, v_i} | > \tau ]$ is bounded by
\[
\PP[\max\{|f_0|, |f_m|\}\geq\lfloor \tau/2\rfloor^{1/16} \lfloor r/4\rfloor^{2/3}-1]
+ \PP[\cC_{\lfloor \tau/2\rfloor, \lfloor r/2 \rfloor}] + \PP[\cC_{\lfloor \tau/2\rfloor, \lfloor (s-r)/2 \rfloor}].
\]
Note that $\ell_0 < \lfloor \tau/2\rfloor < \lfloor r/4\rfloor^{0.01} \leq \lfloor r/2 \rfloor^{0.01}, \lfloor (s-r)/2 \rfloor^{0.01}$,
so by Proposition \ref{prop:numnoncoalgeo} we have
\begin{equation}   \label{eq:upbound:2}
\PP[\cC_{\lfloor \tau/2\rfloor, \lfloor \frac{r}{2} \rfloor}] + \PP[\cC_{\lfloor \tau/2\rfloor, \lfloor (s-r)/2 \rfloor}]
< 2e^{-c_1\tau^{1/8}},
\end{equation}
for some constant $c_1>0$.
By Proposition \ref{prop:spaflu}, and translation invariance, for some constant $c_2>0$ we have
\[
\PP[|f_0|>\lfloor \tau/2\rfloor^{1/16} \lfloor r/4\rfloor^{2/3}/2] ,\; \PP[|f_m - md|>\lfloor \tau/2\rfloor^{1/16} \lfloor r/4\rfloor^{2/3}/2] < e^{-c_2\tau^{1/16}}.
\]
Since $r$ is taken large enough, and $md < r^{2/3}$, $\tau > 10^{10}$, we get that 
\[
\PP[\max\{|f_0|, |f_m|\}\geq\lfloor \tau/2\rfloor^{1/16} \lfloor r/4\rfloor^{2/3}-1] < 2e^{-c_2\tau^{1/16}}.
\]
Using this and \eqref{eq:upbound:2}, we have that 
\[
\PP[| \LL_r \cap \cup_{i=0}^m \Gamma_{u_i, v_i} | > \tau ] < 2e^{-c_2\tau^{1/16}} + 2e^{-c_1\tau^{1/8}}.
\]
Finally, note that $| \LL_r \cap \cup_{i=0}^m \Gamma_{u_i, v_i} | \leq m+1 < r^{2/3}+1$, we have
\[
\begin{split}
&\E[ | \LL_r \cap \cup_{i=0}^m \Gamma_{u_i, v_i} | ]
\\
< &
\tau_0 
+ (r^{2/3}+1)\PP[| \LL_r \cap \cup_{i=0}^m \Gamma_{u_i, v_i} | > r^{0.001} ]
+ \tau_0\sum_{\tau=\lceil \tau_0\rceil}^{\lceil r^{0.001} \rceil} \PP[| \LL_r \cap \cup_{i=0}^m \Gamma_{u_i, v_i} | > \tau ]
\\
< &
\tau_0 
 + (r^{2/3}+1)(2e^{-c_2(r^{0.001})^{1/16}} + 2e^{-c_1(r^{0.001})^{1/8})}
+ \tau_0\sum_{\tau=\lceil \tau_0\rceil}^{\lceil r^{0.001} \rceil} 2e^{-c_2\tau^{1/16}} + 2e^{-c_1\tau^{1/8}},
\end{split}
\]
and this is upper bounded by a constant.
\end{proof}

\appendix

\section{Transversal fluctuation estimates and disjoint geodesics}
In this appendix we provide proofs of Proposition \ref{prop:spaflu} and \ref{prop:numnoncoalgeo}, following arguments in the proof of \cite[Theorem 3]{basu2017coalescence} and \cite[Theorem 2]{basu2018nonexistence}, respectively.

We start with estimates on passage times.
We have that $T_{\boo,(m,n)}$ has the same law as the largest eigenvalue of $X^*X$ where $X$ is an $(m+1)\times (n+1)$ matrix of i.i.d.\ standard complex Gaussian entries (see \cite[Proposition 1.4]{johansson2000shape} and \cite[Proposition 1.3]{BGHK19}).
Hence we get the following one point estimates from \cite[Theorem 2]{LR10}. 
\begin{theorem}
\label{t:onepoint}
There exist constants $C,c>0$, such that for any $m\geq n \geq 1$ and $x>0$, we have
\begin{equation}  \label{e:wslope}
\PP[T_{\boo, (m,n)}-(\sqrt{m}+\sqrt{n})^{2} \geq xm^{1/2}n^{-1/6}] \leq Ce^{-cx}.    
\end{equation}
In addition, for each $\psi>1$, there exist $C',c'>0$ depending on $\psi$ such that if $\frac{m}{n}< \psi$, we have
\begin{equation}  \label{e:slope}
\begin{split}
&\PP[T_{\boo, (m,n)}-(\sqrt{m}+\sqrt{n})^{2} \geq xn^{1/3}] \leq C'e^{-c'\min\{x^{3/2},xn^{1/3}\}},\\
&\PP[T_{\boo, (m,n)}-(\sqrt{m}+\sqrt{n})^{2} \leq -xn^{1/3}] \leq C'e^{-c'x^3},
\end{split}
\end{equation}
and as a consequence
\begin{equation}
\label{e:mean}
|\E T_{\boo, (m,n)} -(\sqrt{m}+\sqrt{n})^2|\leq C'n^{1/3}.
\end{equation}
\end{theorem}
We also have the following segment to segment estimate.
\begin{proposition}
\label{t:treetilted}
Let $A$ and $B$ be segments of length $n^{2/3}$ which are aligned with $\LL_0$ and $\LL_{2n}$ respectively, and let their midpoints being $(m,-m)$ and $\bn$.
For each $\psi<1$, there exist constants $C,c>0$ depending only on $\psi$, such that when $|m|<\psi n$,
\[
\PP\Big[\sup_{u\in A, v\in B} |T_{u,v}-\E T_{u,v}| \geq x n^{1/3}\Big]\leq Ce^{-cx}.
\]
\end{proposition}
A more general version of this result (where $u, v$ are taken in a parallelogram rather than two segments) is proved as \cite[Proposition 10.1, 10.5]{basu2014last}, in the setting of Poissionian DLPP (see also \cite[Proposition B.1]{hammond2020modulus}).
In the setting of exponential DLPP a proof is given in \cite[Appendix C]{basu2019temporal} by Basu, Ganguly, and the author, following the ideas in \cite{basu2014last}.
For completeness we also write the proof here, and it is reproduced from the proof in \cite[Appendix C]{basu2019temporal}.

We start with a segment to point lower bound.
\begin{lemma}
\label{l:infpointtoside}
Let $A'$ denote the line segment of length $2n^{2/3}$ on $\LL_0$ with midpoint at $(m,-m)$.
For each $\psi<1$, there exist $C,c>0$ depending only on $\psi$ such that when $|m|<\psi n$, we have for all $x>0$ and all $n\geq 1$,
\[\PP\Big[ \inf_{u\in A'}  (T_{u,\bn}-\E T_{u,\bn}) \leq -xn^{1/3}\Big]\leq Ce^{-cx^3}.\]
\end{lemma}
\begin{proof}
For simplicity of notations, we write this proof only for the case $m=0$, and the same proof will apply to the general case.
We shall also ignore some rounding issues.
We will use $C,c>0$ to denote large and small universal constants throughout this proof, and the value could change from line to line.

We construct a tree $\mathcal{T}$ whose vertices are a subset of vertices of $\Z^2$; in particular root of $\mathcal{T}$ will be the vertex $\bn$ and the leaves of $\mathcal{T}$ are the vertices on $A'$. Let $n$ be sufficiently large so that there exists $J$ such that $n^{1/4}< 8^{-J}(2n)\leq n^{1/3}$.
For smaller $n$ the lemma follows by taking $C$ large and $c$ small enough. For $j=0,1,2,\ldots, J$, there will be $4^{j}$ vertices of $\mathcal{T}$ at level $j$ (let us denote this set by $\mathcal{T}_j$) on the line $\LL_{8^{-j}(2n)}$, such that these $4^{j}$ vertices divide the line joining $8^{-j}\bn +(-n^{2/3},n^{2/3})$ and $8^{-j}\bn +(n^{2/3},-n^{2/3})$ into $4^{J}+1$ equal length intervals. Notice that, for each $j$ the vertices in $\mathcal{T}_j$ are ordered under $\prec$ from left to right. The vertices of $\mathcal{T}$ is $\cup_{0\leq j\leq J} \mathcal{T}_j$, and the $k$-th vertex at level $j$ from the left is connected to the four vertices in level $(j+1)$ which are labeled $4k-3,4k-2,4k-1$ and $4k$ from the left. 

Noticing that it suffices to prove this lemma for $x$ sufficiently large.
Let $\cA_j$ denote the event that for all $u\in \mathcal{T}_j$ and for all $v\in \mathcal{T}_{j+1}$ such that the edge $(u,v)$ is present in $\mathcal{T}$, we have $T_{v,u}\geq \E T_{v,u}- x2^{-9j/10-10}n^{1/3}$. 
We claim that $\PP(\cup_{j}\cA_j^c) \leq Ce^{-cx^3}$ for all $x$ sufficiently large.
Indeed, by our construction of $\mathcal{T}$, for each edge between a vertex  $u\in \mathcal{T}_j$ and a vertex $v\in \mathcal{T}_{j+1}$, using \eqref{e:slope} we have that 
\[\PP[T_{v,u}-\E T_{v,u}\leq -x{2^{-9j/10-10}}n^{1/3}]\leq Ce^{-cx^32^{3j/10}}.\]
By taking a union bound over all $4^{j+1}$ such edges, and over all $j=0,1,2,\ldots, J-1$, we get that $\PP(\cup_{j}\cA_j^c) \leq Ce^{-cx^3}$.

Now it remains to show that, on $\cap_{0\leq j\leq J} \cA_j$, we have $\inf_{u\in A'}  (T_{u,\bn}-\E T_{u,\bn}) \geq -xn^{1/3}$ for all $x$ sufficiently large. 
Let us fix $u\in A'$ and let $u^{(J)}$ be its closest vertex in $\mathcal{T}_J$. Let $u<u^{(J)}< u^{(J-1)}<\ldots< u^{(0)}=\bn$ denote the path to $\bn$ in $\mathcal{T}$. Hence we have 
$T_{u,\bn} \geq \sum_{j=0}^{J-1} T_{u^{(j+1)},u^{(j)}}$.
By definition, on $\cap_{0\leq j\leq J} \cA_j$ we also have that 
$\sum_{j=0}^{J-1} T_{u^{(j+1)},u^{(j)}}- \E T_{u^{(j+1)},u^{(j)}} \geq -\frac{x}{2} n^{1/3}$
for $x$ sufficiently large. Observe also that by our definition of $\mathcal{T}$, we have $\E T_{u,u^{(J)}}\leq \frac{x}{4}n^{1/3}$ by \eqref{e:mean}, and hence it suffices to lower bound $\sum_{j=0}^{J} \E T_{u^{(j+1)},u^{(j)}}$ (here we write $u=u^{(J+1)}$). 
For $u=(u_1,u_2)\in \Z^2$, we denote $\phi(u)=u_1-u_2$. Observe now that, by the construction of $\mathcal{T}$ and the choice of $u^{(J)}$, we have for each $j\leq J$ that $| \phi(u^{(j+1)})-\phi(u^{(j)})|\leq \frac{Cn^{2/3}}{4^{j}}$. Using \eqref{e:mean} we get $\E T_{u^{(j+1)},u^{(j)}} \geq 2(d(u^{(j)})-d(u^{(j+1)}))-C2^{-j}n^{1/3}$, for each $j\leq J$.
Summing over $j$ from $0$ to $J$ we get $\sum_{j=0}^{J} \E T_{u^{(j+1)},u^{(j)}} \geq \E T_{u,\bn}-\frac{x}{4}n^{1/3}$, and this completes the proof.
\end{proof}
\begin{proof}[Proof of Proposition \ref{t:treetilted}]
Applying Lemma \ref{l:infpointtoside} twice we get the bound for the lower tail of $\inf_{u\in A, v\in B} T_{u,v}-\E T_{u,v}$.
Now denote
$\cA_0=\{\sup_{u\in A,v\in B}  (T_{u,v}-\E T_{u,v}) \geq xn^{1/3}\}$, it remains to show that
$\PP[\cA_0]\leq Ce^{-cx}$
for some $C,c>0$ depending only on $\psi$.
Again we only prove the case $m=0$ for simplicity of notations, and the general case follows similarly. We also observe that it suffices to prove the result for $x$ sufficiently large. 
Consider the following events:
\[
\cA_1=\{\inf_{u\in A} T_{-\bn,u}-\E T_{-\bn,u} \geq -\frac{xn^{1/3}}{10}\}, \;
\cA_2=\{\inf_{v\in B} T_{v,2\bn}-\E T_{v,2\bn} \geq -\frac{xn^{1/3}}{10}\},
\]
It follows from \eqref{e:mean} that for $x$ sufficiently large we have for any $u\in A$ and $v\in B$, we have
$\E T_{-\bn,u}+\E T_{u,v}+ \E T_{v,2\bn} \geq \E T_{-\bn,2\bn} -\frac{xn^{1/3}}{10}$.
It therefore follows that 
$\cA\supseteq \cA_0\cap \cA_1 \cap \cA_2$, where 
$\cA=\{T_{-\bn,2\bn}-\E T_{-\bn,2\bn} \geq \frac{xn^{1/3}}{2}\}$.
Since $\cA_0,\cA_1,\cA_2$ are all increasing events in the weights $\{\xi_v\}_{v\in\Z^2}$, it follows that 
$\PP[\cA]\geq \PP[\cA_0\cap \cA_1\cap \cA_2]\geq \PP[\cA_0]\PP[\cA_1]\PP[\cA_2]$
by the FKG inequality. The result follows by noting that we have $\PP[\cA_1], \PP[\cA_2]\geq \frac{1}{2}$ by Lemma \ref{l:infpointtoside}, and $\PP[\cA]\leq C'e^{-c'\min\{x^{3/2},xn^{1/3}\}}$ by \eqref{e:slope}.
\end{proof}

\subsection{Transversal fluctuation}
Now we give the proof of Proposition \ref{prop:spaflu}.
We note that it is an analog of \cite[Theorem 3]{basu2017coalescence}, in the sense that both bound the fluctuation of the geodesic from the diagonal.
The main difference is that in \cite[Theorem 3]{basu2017coalescence}, the fluctuation is measured in terms of the intersections with a vertical or horizontal line, while here we use the perpendicular distance of a point on the geodesic to the diagonal.
In our setting the transversal fluctuation at $\LL_{2r}$ is at most linear in $r$, avoiding some technical complexities (see also \cite[Remark 1.3]{basu2017coalescence}).
\begin{proof}[Proof of Proposition \ref{prop:spaflu}]
For simplicity of notations, we assume that $m=0$, and we write $\Gamma=\Gamma_{\boo, \bn}$.
We also assume that $n=2^Jr$ for some $J\in\Z_+$. The general case follows the same arguments.
By symmetry it suffices to show that $\PP[f_0 > xr^{2/3}] < e^{-cx}$.

Let $\alpha=2^{\frac{1}{6}}$.
For $0\leq j \leq J$, take $f_j\in\Z$ such that $(2^{j-1}r+f_0, 2^{j-1}r-f_0) = \Gamma \cap \LL_{2^jr}$. Denote $f_{J+1}=0$.
Let $\cB_j$ denote the event that $f_j > x((2\alpha)^jr)^{2/3}$ and $f_{j+1} \leq x((2\alpha)^{j+1}r)^{2/3}$.
Then $\{f_0 > xr^{2/3}\} \subset \cup_{j=0}^{J-1} \cB_j$, and it suffices to show that 
$\PP[\cB_j]\leq e^{-cx\alpha^{2j/3}}$.
For this we split $\LL_{2^jr}$ and $\LL_{2^{j+1}r}$ into segments of length $(2^jr)^{2/3}$.
For $t,t'=0,1,2,\ldots$, denote
\[
\begin{split}
U_t=\{(2^{j-1}r+f, 2^{j-1}r-f): f\in (x((2\alpha)^jr)^{2/3} + t(2^jr)^{2/3}, x((2\alpha)^jr)^{2/3} + (t+1)(2^jr)^{2/3}]\},\\
V_{t'}=\{(2^{j}r+f, 2^{j}r-f): f\in[x((2\alpha)^{j+1}r)^{2/3}-(t+1)(2^jr)^{2/3}, x((2\alpha)^{j+1}r)^{2/3}-t(2^jr)^{2/3})\},
\end{split}
\]
and we let $\cB_{j,t,t'}=\{\exists u\in U_t, v\in V_{t'}, \boo<u<v, T_{\boo,u}+T_{u,v}-T_{\boo,v}\geq 0\}$.
Then we have
$\cB_{j}\subset \cup_{t,t'=0}^{\infty}\cB_{j,t,t'}$.
We claim that $\PP[\cB_{j,t,t'}]\leq e^{-c_1(x\alpha^{\frac{2j}{3}}+t+t')}$ for some $c_1>0$.
Then summing over $j,t,t'$ gives the result.

Fix some $j,t,t'$, and assume that there exist $u\in U_t$ and $v\in V_{t'}$ with $\boo<u<v$.
Then for any $v=(v_1,v_2)\in V_{t'}$ we must have that $\frac{v_1}{v_2}\in (0.1, 10)$.
We can also assume that for all $u=(u_1,u_2)\in U_t$ and $v=(v_1,v_2)\in V_{t'}$, we have $\frac{u_2}{u_1}, \frac{v_2-u_2}{u_2-u_1}\in (0.01,100)$ (otherwise, we can get the desired bound of $\PP[\cB_{j,t,t'}]$ by \eqref{e:wslope} and \eqref{e:slope} and taking a union bound over all $u\in U_t$, $v\in V_{t'}$ with $\boo<u<v$).
Now by \eqref{e:mean}, we have that there exists some constant $C$ independent of $r,x,t,t',j$, such that for all $u\in U_t, v\in V_{t'}$ and all $x$ sufficiently large,
we have $\E T_{\boo,u}+\E T_{u,v}\leq \E T_{\boo,v}-C(x\alpha^{\frac{2j}{3}}+t+t')^2(2^jr)^{1/3}$.
By Proposition \ref{t:treetilted} we get  $\PP[\cB_{j,t,t'}]\leq e^{-c_1(x\alpha^{\frac{2j}{3}}+t+t')}$ as desired.
\end{proof}

\subsection{Disjoint geodesics}
For completeness, in this subsection we give a proof of Proposition \ref{prop:numnoncoalgeo} (which is precisely \cite[Corollary 2.7]{basu2018nonexistence}). It is reproduced from the proof of \cite[Theorem 2]{basu2018nonexistence}, and the only difference is that we write in the slightly more general setting of \cite[Corollary 2.7]{basu2018nonexistence}.

Fix $n, \ell$, which are sufficiently large and satisfy $\ell < n^{0.01}$.
Below we will ignore some rounding issues.
Let $h=\ell^{1/2}$, consider $h$ line segments $L_{0},\cdots, L_h$, each with slope $-1$ and length $2\ell ^{1/8}n^{2/3}$, and the midpoints are in the line connecting $\boo$ and $\bn$.
These segments are equally spaced with internal spacing $\frac{2n}{h}$, and $L_{0}\subset\LL_{0}$, $L_{h}\subset\LL_{2n}$.

Let $\cE$ be the event that there exist $u\in A_{\ell,n}$, $v\in B_{\ell,n}$, and $0\leq i \leq h$, such that $\Gamma_{u,v}$ does not intersect $L_i$.
Then we have $\PP(\cE)\leq e^{-c\ell^{1/8}}$ for some $c>0$.
Indeed, by Lemma \ref{l:ordered}, $\cE$ implies that at least one of $\Gamma_{u_l,u_l+\bn}$ and $\Gamma_{u_r,u_r+\bn}$ do not intersect some $L_i$.
Here $u_l$ and $u_r$ are two end points on $A_{\ell,n}$.
Then we just apply Proposition \ref{prop:spaflu}.

Take a constant $c_1>0$ to be determined.
Using \eqref{e:mean} and Proposition  \ref{t:treetilted}, we have that 
$\PP[\inf_{u\in A_{\ell,n}, v\in B_{\ell,n}} T_{u,v} \leq 4n-c_1\ell^{1/4}n^{1/3}] \leq  e^{-c\ell^{1/4}}$
for some $c>0$. Now to get Proposition \ref{prop:numnoncoalgeo} it suffices to prove the following result.
\begin{proposition}
\label{p:multi}
Let $\cg_{\ell,n}$ denote the event that there exists $\tu_1, \cdots , \tu_{\ell}\in A_{\ell,n}$ and $\tv_1, \cdots , \tv_{\ell}\in B_{\ell,n}$, and disjoint paths $\gamma_{i}$ joining $\tu_i$ and $\tv_i$, such that $\gamma_i$ intersects $L_k$ for each $0\leq k\leq h$,
and $T(\gamma_{i}) \geq 4n-c_1\ell^{1/4} n^{1/3}$.
Then $\PP(\cg_{\ell,n})\leq e^{-c\ell^{1/4}}$.
\end{proposition}
Our argument to prove this proposition is as follows.
Take some $g\in\Z_+$ to be determined, and divide the line segment $L_i$ into equally spaced line segments $L_{i,j}$ (for $j \in [-\ell^{1/8}g, \ell^{1/8}g)\cap \Z$), each of length $\frac{n^{2/3}}{g}$. 
For a fixed sequence $J:=\{j_0, j_1,j_2,\ldots , j_{h-1}, j_{h}\}$ taking values in $[-\ell^{1/8}g, \ell^{1/8}g)\cap \Z$, we shall consider the best path $\gamma_{J}$ from $A_{\ell,n}$ to $B_{\ell,n}$ that passes through the line segment $L_{i,j_{i}}$ for each $i=0,1,\ldots, h$. We shall show that $T(\gamma_{J})$ is typically much smaller than $4n$. For this we need the following  lemma.
\begin{lemma}
\label{l:mean}
Let $A_*$ denote the line segment joining $(-\frac{c_0n^{2/3}}{2}, \frac{c_0n^{2/3}}{2})$ and $(\frac{c_0n^{2/3}}{2}, -\frac{c_0n^{2/3}}{2})$, and let $B_*=A_*+(n+m,n-m)$. For $c_0$ sufficiently small there exists $c_2>0$ such that for all $n$ sufficiently large and $|m|<0.9n$, we have
$\E \sup_{u\in A_*, v\in B_*} T_{u,v} \leq 4n-c_2n^{1/3}$.
\end{lemma}

\begin{proof}
For simplicity of notations we assume that $m=0$, and the same arguments work for the general case.
Consider $u_0= (-c_0^{3/2}n, -c_0^{3/2}n)$, $v_0=\bn+(c_0^{3/2}n, c_0^{3/2}n)$, and note
\[\E \sup_{u\in A_*, v\in B_*} T_{u,v} \leq  \E T_{u_0,v_0} -\E\inf_{u\in A_*} T_{u_0,u}-\E\inf_{v\in B_*} T_{v,v_0}.\]
By Proposition \ref{t:treetilted} we have $\E\inf_{u\in A_*} T_{u_0,u}, \E\inf_{v\in B_*} T_{v,v_0} \geq 4c_0^{3/2}n-Cc_0^{1/2}n^{1/3}$ for some constant $C>0$. From the Tracy-Widom convergence result of \cite{johansson2000shape}, and the fact that the GUE Tracy-Widom distribution has negative mean, it follows that that for $n$ sufficiently large $\E T_{u_0,v_0}\leq 4(n+2c_0^{3/2}n)-C'n^{1/3}$ for some $C'>0$. Thus we can choose $c_0$ sufficiently small to complete the proof.
\end{proof}
From now on we fix $c_0$ such that Lemma \ref{l:mean} holds, and choose $g=\frac{h^{2/3}}{c_0}$. 
\begin{lemma}
\label{l:thin1}
For any sequence $J:=\{j_0, j_1,j_2,\ldots , j_{h-1}, j_{h}\}$ taking values in $[-\ell^{1/8}g, \ell^{1/8}g)\cap \Z$, let $\gamma_{J}$ denote the best path from $A_{\ell,n}$ to $B_{\ell,n}$ that intersects the line segment $L_{i,j_{i}}$ for each $i=0,1,\ldots, h$. 
Then there exists  $c_1,c>0$, such that for each $J$ we have
$\PP[T(\gamma_{J})\geq 4n-c_1 h^{2/3}n^{1/3}]\leq e^{-ch^{1/2}}$.
\end{lemma}

\begin{proof}
Fix a sequence $J$, first observe that $T(\gamma_{J})\leq \sum_{i=0}^{h-1} \sup_{u\in L_{i,j_i}, v\in L_{i+1,j_{i+1}}} T_{u,v}$,
which is a sum of $h$ independent random variables.
By our choice of $\ell$ and $h$ we have that the slope of the line joining the midpoints of $L_{i,j_i}$ and $L_{i+1,j_{i+1}}$ is between $1/2$ and $2$. Thus for each $i$, by Lemma \ref{l:mean} we have 
\[\E \sup_{u\in L_{i,j_i}, v\in L_{i+1,j_{i+1}}} T_{u,v}\leq  4n/h - c_2(n/h)^{1/3};\]
and by Proposition \ref{t:treetilted}, for large enough $x$ we have
\[\PP\Big[ \sup_{u\in L_{i,j_i}, v\in L_{i+1,j_{i+1}}} T_{u,v}-4n/h \geq x(n/h)^{1/3}\Big] \leq e^{-cx}.\]
By a Bernstein type bound on sum of independent variables with exponential tails, we have
$\PP[ T(\gamma_{J}) -(4n-c_2 h^{2/3}n^{1/3}) \geq x h^{1/6}n^{1/3}]\leq e^{-cx}$
for some $c>0$ and each $x>0$ sufficiently large. By taking $c_1=c_2/2$ and $x=c_1h^{1/2}$ we finish the proof.
\end{proof}

\begin{proof}[Proof of Proposition \ref{p:multi}]
We use the BK inequality, and bound the entropy term of disjoint paths.
Clearly we can assume that $n, \ell$ are sufficiently large with $\ell<n^{0.01}$.

We let $\mathcal{C}$ denote the set of all tuples $(J_1,J_2,\ldots , J_{\ell})$, where each $J_{i}=(j^{(i)}_{0}, \ldots, j^{(i)}_{h})$ is a sequence taking values in $[-\ell^{1/8}g, \ell^{1/8}g)\cap \Z$, satisfying that $j_k^{(i_1)}\leq j_k^{(i_2)}$ for any $1\leq i_1<i_2\leq \ell$ and $0\leq k\leq h$.
For $(J_1,J_2,\ldots , J_{\ell})\in \mathcal{C}$, let $\cA_{J_1,J_2,\ldots , J_{\ell}}$ denote the event that there exist disjoint paths $\gamma_1, \ldots, \gamma_{\ell}$ from $A_{\ell,n}$ to $B_{\ell,n}$, such that each $T(\gamma_{i}) \geq 4n-c_1h^{2/3}n^{1/3}$,
and $\gamma_i$ intersects $L_{k,j_k^{(i)}}$ for each $k$.
Since geodesics are ordered (Lemma \ref{l:ordered}), we have that $\cg_{\ell,n} \subset \cup_{(J_1,J_2,\ldots , J_{\ell})\in \mathcal{C}} \cA_{J_1,J_2,\ldots , J_{\ell}}$.

For each $1\le i\le \ell$, consider the event that there is a path $\gamma_i$ from $A_{\ell,n}$ to $B_{\ell,n}$, intersecting $L_{k,j_k^{(i)}}$ for each $k$, with $T(\gamma_i)\geq 4n-c_1h^{2/3}n^{1/3}$. This is an increasing event in the weights $\{\xi_v\}_{v\in\Z^2}$. Thus by the BK inequality \cite{van1985inequalities} \footnote{The typical BK inequality \cite[Theorem 3.3]{van1985inequalities} is stated for the case where the weights at vertices are i.i.d. Bernoulli. For our exponential weight setting we can directly apply \cite[Theorem 1.6(iii)]{van1985inequalities}.}, the probability of such events happening disjointly is upper bounded by the product of the marginal probabilities. It therefore follows using Lemma \ref{l:thin1} that $\PP(\cA_{J_1,J_2,\ldots , J_{\ell}})\leq e^{-c\ell^{5/4}}$.

It remains to bound $|\mathcal{C}|$.
Note that $\mathcal{C}$ can be enumerated by picking $(h+1)$ many non-decreasing sequences of length $\ell$, where each co-ordinate takes values between $-\ell^{1/8}g$ and $\ell^{1/8}g$. Now we need to enumerate positive integer sequences $-\ell^{1/8}g \leq y_1 \leq y_2 \leq \cdots \leq  y_{\ell} \leq \ell^{1/8}g$. By taking the difference sequence $z_{k}=(y_{k}-y_{k-1})$ this reduces to enumerating sequences with $z_1+z_2+\cdots +z_{\ell} \leq 2\ell^{1/8}g$. It is a standard counting exercise to see that number of such sequences is bounded by $\binom {\ell+2\ell^{1/8}g}{2\ell^{1/8}g}$.
Hence as $h=\ell^{1/2}$ and $g=\frac{h^{2/3}}{c_0}$, for any $\epsilon>0$ we have $|\mathcal{C}|\leq \ell^{\ell^{23/24+\epsilon}}$ and the result follows.
\end{proof}

%%%%%%%%%%%%%%%%%%%%%%%%%%%%%%%%%%%%%%%%%%%%%%%%%%%%%%%%%%%%%%%%%%%
%%                                                               %%
%% Use the two commands below for producing your bibliography    %%
%% with bibtex, then comment again the commands and include the  %%
%% content of the .bbl file in this file below the commands.     %%
%%                                                               %%
%%%%%%%%%%%%%%%%%%%%%%%%%%%%%%%%%%%%%%%%%%%%%%%%%%%%%%%%%%%%%%%%%%%

%\bibliographystyle{amsplain}
%\bibliography{yourbibfilename}

% add below the content of your .bbl file produced by bibtex.

\providecommand{\bysame}{\leavevmode\hbox to3em{\hrulefill}\thinspace}

%%%%%%%%%%%%%%%%%%%%%%%%%%%%%%%%%%%%%%%%%%%%%%%%%%%%%%%%%%%%%%%%%%%
%%                                                               %%
%% You may add acknowledgments (optional).                       %%
%%                                                               %%
%%%%%%%%%%%%%%%%%%%%%%%%%%%%%%%%%%%%%%%%%%%%%%%%%%%%%%%%%%%%%%%%%%%

\ACKNO{The author would like to thank his advisor Professor Allan Sly for telling him this problem, and for reading early versions of this paper. The author also thanks anonymous referees for reading this paper carefully, and for their valuable feedback which led to many improvements in the text.}

\end{document}